\documentclass{amsart}
\usepackage{cite}
\usepackage{amsthm}
\title{Recurrence Relations for Exceptional Hermite Polynomials}

\author{D. G\'omez-Ullate} \address{Departamento de F\'isica Te\'orica II, Universidad Complutense de
Madrid, 28040 Madrid, Spain.}
\address{Instituto de Ciencias Matem\'aticas (CSIC-UAM-UC3M-UCM),  C/ Nicolas Cabrera 15, 28049 Madrid, Spain.}

\author{A. Kasman} \address{Department of Mathematics, College of
Charleston, Charleston SC, USA}

\author{A.B.J.~Kuijlaars}
\address{Department of Mathematics, KU Leuven, Belgium}

\author{R. Milson} \address{Department of Mathematics,  Dalhousie University, Halifax, Nova Scotia, Canada}

\newcommand{\stabf}{f}
\newcommand{\Deltahat}{\hat\Delta_{\stabf}}
\newcommand{\tDelta}{\tilde\Delta_{\stabf}}
\newcommand{\C}{\mathbb{C}}
\newcommand{\N}{\mathbb{N}} 
\newcommand{\R}{\mathbb{R}} 
 
\newcommand{\shift}{\Theta}

\newcommand{\tildh}{\tilde{h}}
\newcommand{\tf}{\tilde{f}}

\newcommand{\rL}{\mathrm{L}}
\newcommand{\LRWlam}{\rL^2(\R,\hW(x)dx)}
\newcommand{\cA}{\relax{A}}
\newcommand{\cB}{\relax{B}}

\newcommand{\cH}{\mathcal{H}}

\newcommand{\StabRing}{\mathcal{S}}
\newcommand{\Ilam}{\mathcal{I}}
\newcommand{\Wlam}{\hat W}
\newcommand{\cU}{\mathcal{U}}
\newcommand{\hh}{\hat{h}}

\newcommand{\hW}{\hat{W}}

\newcommand{\hT}{\hat{T}}

\newcommand{\Wr}{\operatorname{Wr}}
\newcommand{\Wrop}[1]{\Wr\left[ #1\right]}

\newcounter{theoremct}[section]
 
\newtheorem{corollary}[theoremct]{Corollary}
\newtheorem{lemma}[theoremct]{Lemma}

\newtheorem{prop}[theoremct]{Proposition}
\newtheorem{theorem}[theoremct]{Theorem}

\theoremstyle{definition} \newtheorem{definition}[theoremct]{Definition}

\theoremstyle{remark} \newtheorem{remark}[theoremct]{Remark}

\numberwithin{equation}{section}

\begin{document}
\begin{abstract}
The bispectral anti-isomorphism is applied to differential operators involving elements of the stabilizer ring to produce explicit formulas for all difference
  operators having any of the Hermite exceptional orthogonal polynomials as eigenfunctions with eigenvalues that are polynomials in $x$. 
	\end{abstract}

\maketitle

\section{Introduction}

The classical orthogonal polynomial families of Hermite, Laguerre and
Jacobi \cite{Szego1939} have three important properties: they are
eigenfunctions of a differential operator in $x$ having eigenvalues
dependent on $n$, they satisfy a well-known three-term recurrence
relation, and they form a basis of a given $L^2$ space.

In the past few years, the class of polynomials having the above
properties has been extended to include exceptional orthogonal
polynomials, which differ from the classical families in that there
are a finite number of degrees for which no polynomial eigenfunction
exists.  The differential equation for exceptional polynomials
contains rational instead of polynomial coefficients.  
If certain regularity assumptions are imposed then, like their classical counterparts, the exceptional orthogonal polynomials form the basis of a weighted $L^2$ space.

This paper is mostly concerned with the study of recurrence relations
for exceptional Hermite polynomials.  Some particular examples of
these polynomials were investigated as early as \cite{Dubov1994}, but
a final classification appeared only recently in
\cite{Gomez-Ullate2014a}.  See that paper for a more detailed overview
of the subject and for additional references. 
Exceptional Hermite polynomials were also derived from  exceptional Charlier polynomials by taking a suitable limit \cite{Duran2014}.

Exceptional polynomials admit two different sets of recurrence
relations.  The first set of recurrence relations 
\cite{Odake2013a,Gomez-Ullate2014a} have coefficients that are polynomial
functions of $x$ and $n$ and consequently have no obvious bispectral
interpretation; i.e., they cannot be interpreted as higher-order
Jacobi operators. We shall not investigate these relations in this paper.  The second set of recurrence relations have  order $2N+3$ where
$N$ is the codimension (number of missing degrees) of the exceptional
family. Except for one term, the coefficients do not depend on $x$ and
are rational functions of $n$. This type of relations will be the main
focus of our study.

The existence of this second type of recurrence relations of order
$4N+1$ was first established in \cite{Sasaki2010}.  Relations of order $2N+3$
order  for exceptional families obtained through a one-step
Darboux transformation were recently given in \cite{Miki2015}, and
later generalized in \cite{Odake2015}, but in both cases an explicit
construction of the recurrence coefficients was missing. Somewhat
closer to this goal is the recent work of Dur\'an \cite{Duran2015a},
which provides explicit relations for some examples of exceptional
Hermite and Laguerre polynomials. His approach uses a bispectral
correspondence between discrete Krall polynomials and exceptional
Charlier polynomials and a limit procedure to obtain the recurrence
relations for the exceptional Hermite.

The main result of this note (Theorem~\ref{thm:Deltaf}) is a straight-forward procedure for producing the  general recurrence relations satisfied by the exceptional  polynomials. We focus on the Hermite subclass but it is clear that the same techniques can be applied to the other families as well.  In particular, by making use of a bispectral anti-isomorphism of rings of operators acting on the classical Hermite polynomials, the production of recurrence relations for the exceptional Hermite polynomials is reduced to elementary algebra.

The idea of the bispectral anti-isomorphism for creating ``bispectral
Darboux transformations'' was pioneered by Kasman-Rothstein \cite{Kasman1997}, 
further developed by Bakalov-Horozov-Yakimov \cite{Bakalov1996}, and first used in the context of orthogonal polynomials by Gr\"unbaum-Yakimov \cite{Grunbaum2002}.  The construction
below does not depend on those prior results but is an application of
the same ideas to the new context of exceptional orthogonal polynomials.

\section{Background}

\subsection{Hermite Polynomials}

Set $h(n,x)=H_n(x)$, where the latter is the $n$th degree Hermite
polynomial as defined in Section 18.3 of \cite{NIST}.  These polynomials
have associated operators with the following properties.

Letting
\[ \shift(f(n)) = f(n+1)\] denote the elementary shift operator, define
the following degree-raising and lower operators,
\begin{align} \label{eq:Delta}
  \Delta  &= \frac12 \shift  + n\circ  \shift^{-1}, \\
   \Gamma &= 2n\circ \shift^{-1}, \label{eq:Gamma}
\end{align}
where the $n$ denotes the indicated multiplication operator.  The action of 
$\Delta$ and $\Gamma$ on the classical Hermite polynomials can also be
expressed as operators in $x$, namely
\begin{align}
  \label{eq:3term}
   x h(n,x) &= \Delta h(n,x)  \\
 \partial h(n,x) &= \Gamma h(n,x)
 \end{align}
 where $\partial = \frac{\partial}{\partial x}$ differentiates with respect to $x$.
 Note that \eqref{eq:3term} is just the 3-term recurrence relation for
 the classical Hermite polynomials.  We will also need the
 second-order operator
 \[T(y(x)) = y''(x) - 2x y'(x) \] and express the Hermite differential
 equation as
\begin{equation}T(h(n,x))=-2nh(n,x).\label{eqn:Teigen}\end{equation}

\begin{remark} To read and understand the remainder of the
paper, it is helpful to keep in mind that operators in $x$ only always commute with
operators acting only in $n$.   In general, we
will try to write operators in $x$ with Roman symbols (except for $\partial$) and operators in
$n$ with Greek ones to help the reader remember which variable an operator acts in.  However, the map $\flat$ to be introduced later
turns operators in $x$ into operators in $n$.  Moreover, 
$\flat(R)$ and $R$ have the same effect on $h(n,x)$ for any $R\in\C[x,\partial]$.  So, when either one of these expressions
appears next to $h(n,x)$ in a formula it is possible to replace it
by the other.
\end{remark}

\subsection{Exceptional Hermite Polynomials}\label{sec:XOPintro}

This subsection quickly reviews the definitions related to the Exceptional Hermite Polynomials $\hh(n,x)$ associated to the choice of a partition $\lambda$.  For proofs of the claims and more details, see \cite{Gomez-Ullate2014a}.

Let $\lambda_1\leq \lambda_2\leq \cdots\leq\lambda_{\ell}$ be a
non-decreasing sequence of positive integers so that
$\lambda=\{\lambda_1,\ldots,\lambda_{\ell}\}$ is a partition of the
positive integer $N=\sum \lambda_i$ into $\ell$ parts\footnote{Many of
  the objects in this paper depend on the selection of $\lambda$.
  However, to avoid cumbersome notation, we will consider $\lambda$ to
  be selected and fixed and not specifically indicate the dependence
  of objects, such as $\hh(n,x)$, $\StabRing$, $\cU$ and $\Ilam$ on
  this choice.}.  It is convenient to associate to the choice of
$\lambda$ the set $K=\{ k_1,\ldots, k_\ell \}$, where
\begin{equation}
  \label{eq:kidef}
  k_i = \lambda_i+i-1,\quad  i=1,\ldots,\ell.   
\end{equation}

Using $\Wr$ to denote the Wronskian determinant with respect to $x$, to the selection of partition $\lambda$ (or sequence $K$)
we associate the polynomial $\eta(x)$ defined by
\begin{equation}
  \label{eq:etadef}
  \eta(x)=\Wrop{h(k_1,x),\ldots,h(k_\ell,x)},
\end{equation}
and define the order $\ell$ differential operator $\cA$ by its action on an
arbitrary function $y(x)$,
\begin{equation}
  \label{eq:Aydef}
  \cA(y(x))=\Wrop{h(k_1,x),\ldots,h(k_\ell,x),y(x)}.
\end{equation}

Now, a new sequence of polynomials is defined by 
\begin{equation}
  \hh(n,x):= \cA(
  h({n+\ell- N },x)),
\end{equation}
where it is clear from \eqref{eq:Aydef} that $\hh(n,x)$ is either a degree $n$ polynomial in $x$ or zero.

The degree sequence of the resulting non-zero polynomials is
\begin{equation}
  \label{eq:Ilamdef}
  \Ilam = \{ n \in \N_0 : n \geq  N  -\ell, n+\ell- N 
  \notin K \}. 
\end{equation}
The degrees $0,1,\ldots , N -\ell-1$ and the degrees $k_1 + N
-\ell,\ldots, k_\ell + N -\ell$ are missing from $\Ilam$, so that the
polynomial sequence $\{\hh(n,x):n\in\Ilam\}$ is missing a total of $ N $
degrees.  Let us also denote by $\cU\subset\C[x]$ the set of finite linear combinations of
exceptional Hermite polynomials:
\begin{equation}
\cU=\operatorname{span} \{ \hh(n,x) \colon n\in\Ilam\}
\end{equation}

The resulting families of polynomials generalize the classical Hermite
family because they satisfy a Hermite-like differential equation
\begin{equation}
   \label{eq:XDE}
    \hT(\hh(n,x)) = 2(N-n) \hh(n,x),\qquad n\in \Ilam
    \end{equation}
    with
\begin{equation}
  \hT(y):=y''-2\left(x+\frac{\eta'}{\eta} \right)
  y'+ \left(\frac{\eta''}{\eta}+ 2 x
    \frac{\eta'}{\eta} \right) y.\label{eqn:hatT}
\end{equation}

We will say that $\lambda$ is an even partition if $\ell$ is even and
if $\lambda_{2i-1} = \lambda_{2i}$ for every $i\leq \ell/2$.  For an
even partition $\lambda$, the exceptional Hermite polynomials
$\hh(n,x)$ satisfy the orthogonality relations\footnote{This formula
  corrects a misprint in equation (52) of \cite{Gomez-Ullate2014a}.}
$$ \int_\R \hh(n,x) \hh(m,x) \Wlam(x) dx =
\delta_{m,n} \sqrt{\pi}\, 2^{j+\ell} j! \prod_{i=1}^{\ell}(j-k_i),\qquad n,m\in\Ilam
$$
where $j = n+\ell-N$, and the weight is defined as
$ \Wlam(x) = \relax{e^{-x^2}}/{\eta(x)^2}$.
Moreover,  if $\lambda$ is an even partition then $\cU$ is  dense in $\LRWlam$.

Despite the importance of even partitions to the role of $\hh(n,x)$ as orthogonal polynomials, the formulas for the recurrence relations that constitute the main result of this paper are valid for an arbitrary choice of $\lambda$ so this distinction will not be further emphasized.

\subsection{Bispectrality}

The bispectral problem, as proposed by Gr\"unbaum and initially
studied by Duistermaat-Gr\"unbaum \cite{Duistermaat1986} seeks to find functions
depending on two variables that are joint eigenfunctions for operators
in each of the two variables separately so that in each case the
eigenvalue depends non-trivially on the other variable.  Specifically, a
\textit{bispectral triple} $(L,\Lambda,\psi)$ consists of a function
$\psi(x,z)$, an operator $L$ in $x$ satisfying $L\psi=p(z)\psi$ and an
operator $\Lambda$ in $z$ satisfying $\Lambda\psi=\pi(x)\psi$.

A great body of research by many authors has investigated
bispectrality for various different kinds of operators, establishing
theorems, classifying examples, finding connections to other areas of
research such as soliton equations, or quantum and classical particle
systems.  (See, for instance,
\cite{Harnad1998} and references therein.)
Additionally, it has long been recognized that the classical orthogonal polynomials
can be interpreted as exhibiting bispectrality if they are viewed as
functions of two variables (one continuous and one discrete) rather
than as a sequence of separate functions of $x$.  The differential
equation in $x$ and the recurrence relation are then the two eigenvalue
equations.  (See, for example,
\cite{Grunbaum1999,spiridonov1998}.)

\section{The Anti-Isomorphism}

The key observation leading to the main result is that the ring
$\C[x,\partial]$ of polynomial coefficient differential operators and
the ring $\C[\Delta,\Gamma]$ of difference operators generated by
$\Delta$ and $\Gamma$ are naturally anti-isomorphic.
\begin{definition} \label{def:flat}
Define the map $\flat$ by its action on monomials:
\begin{eqnarray*}
\flat:\C[x,\partial]&\to&\C[\Delta,\Gamma]\\
x^i\partial^j&\mapsto&\Gamma^j\Delta^i
\end{eqnarray*}
extended linearly to all of $\C[x,\partial]$.
\end{definition}

\begin{theorem}\label{thm:flat}
The map $\flat:\C[x,\partial]\to\C[\Delta,\Gamma]$ 
satisfies
$$
Q(h(n,x))=\flat(Q)h(n,x)\ \forall Q\in\C[x,\partial]
$$
and
$$
\flat(Q\circ R)=\flat(R)\circ \flat(Q)\ \forall Q,R\in\C[x,\partial].
$$
In other words, $\flat$ is an anti-isomorphism that has the property that an operator and its image always have exactly the same effect on the classical Hermite polynomials.
\end{theorem}

\begin{proof} Because $\flat$ extends linearly it is sufficient to check the first property on a monomial of the form $x^i\partial^j$.  Then observe that because $x$ and $\Gamma$ commute one has
$$
x^i\partial^j h(n,x)=x^i\Gamma^jh(n,x)=\Gamma^jx^ih(n,x)=\Gamma^j\Delta^ih(n,x)=\flat(x^i\partial^j)h(n,x).
$$

As a consequence, we have for any $Q$ and $R$ in $\C[x,\partial]$ that
$$
\flat(Q\circ R)h(n,x)=(Q\circ R)h(n,x).
$$
But, note that $$\flat(R)\circ \flat(Q)h(n,x)=\flat(R)(Qh)=Q(\flat(R)h(n,x))
=Q(R(h(n,x)))=(Q\circ R)h(n,x)
$$
also because $Q$ commutes with $\flat(R)$.  Then $\flat(Q\circ R)-\flat(R)\circ \flat(Q)$ is an operator in $\C[\Delta,\Gamma]$ which has the function $h(n,x)$ in its kernel for all $n$.  This implies that it is the zero operator.
\end{proof}

\section{Factorization Lemma and Eigenvalue Equations}

Define another differential operator $\cB$ of order $\ell$ by stating
that for an arbitrary function $y(x)$,
\begin{equation}
  \label{eq:Bydef}
  \cB(y(x))=(-1)^{\lfloor \ell/2\rfloor} e^{x^2} \eta^{-\ell}
  \Wrop{\tildh_{1}(x),\ldots, \tildh_{\ell}(x),e^{-x^2} y(x)}
\end{equation}
where
\begin{equation}
  \label{eq:tildhdef}
  \tildh_{j}(x) :=\Wrop{h({k_1},x),\ldots
  \widehat{h({k_j},x)},\ldots,h({k_\ell},x)}
\end{equation}
with the hat decoration indicating omission from the sequence\footnote{This formula corrects a misprint in equation (27) of \cite{Gomez-Ullate2014a}.}.
The main result of this section is Corollary~\ref{cor:BApi} providing eigenvalue equations for the products $\cA\circ\cB$ and $\cB\circ\cA$.  Unfortunately, the introduction of several technical definitions and lemmas is necessary to obtain it.


For $0\leq j\leq k$, let $\lambda^{(j)}$ denote the
truncated partition $\lambda^{(j)} = (\lambda_1\leq \cdots \leq
\lambda_j)$ and $K^{(j)} = \{ k_1, \ldots, k_{j}\}$ the
corresponding truncated set.  Let
\begin{equation} \label{eq:etaj} 
	\eta_j(x) = \Wrop{h({k_1},x),\ldots, h({k_j},x)}, 
	\end{equation} denote the
corresponding Hermite Wronskians.  For $1\leq j\leq
\ell$, define the first-order differential operators
\begin{align} \label{eq:Aj}
  A_j(y) &:= \frac{\eta_j}{\eta_{j-1}} \left(y' -
    \frac{\eta_j'}{\eta_j}  y\right)= \frac{\Wrop{\eta_j,y}}{\eta_{j-1}},\\
    \label{eq:Bj}
  B_j(y) &:= \frac{\eta_{j-1}}{\eta_j}\left( y' - \left( 2x +
      \frac{\eta_{j-1}'}{\eta_{j-1}}\right)y \right) =
  \frac{\Wrop{e^{x^2} \eta_{j-1},y}}{e^{x^2} \eta_j},
\end{align}
where we define
$ \eta_0=1$.

\begin{prop}\label{prop:fact}
  The above operators are linear factors of the order $\ell$
  operators $\cA$ and $\cB$ defined in \eqref{eq:Aydef} and
  \eqref{eq:Bydef}.  Indeed, the following factorizations hold:
  \begin{align}
    \label{eq:Afactor}
    \cA &=  A_{\ell} \circ \cdots \circ A_2\circ  A_1\\
    \label{eq:Bfactor}
    \cB &= B_1\circ  B_2\circ \cdots\circ B_{\ell}
  \end{align}
\end{prop}
\begin{proof}
  The proof makes use of the following classical identities
  \cite{muir}.  Throughout, the arguments are sufficiently
  differentiable functions of one variable.
  \begin{align}
    \label{eq:wronsk1}
    \Wrop{g f_1, g f_2, \ldots, g f_n} &= g^n \Wrop{f_1, f_2,\ldots, f_n},\\
    \label{eq:wronsk2}
     \Wrop{\Wrop{f_1, \ldots, f_n , g}, \Wrop{f_1,\ldots, f_n, h}} 
    & = \Wrop{f_1,\ldots, f_n} \Wrop{f_1,\ldots, f_n ,g,h},\\
    \label{eq:wronsk3}
    \Wrop{\tf_1,\ldots, \tf_n} & = (-1)^{\lfloor n/2\rfloor}
    \Wrop{f_1,\ldots, f_n}^{n-1},
   \end{align}
   where
  \begin{align} \nonumber
    \tf_j = \Wrop{f_1,\ldots, \hat{f_j},\ldots, f_n}.
  \end{align}

  Let us prove \eqref{eq:Afactor}.  Since $A_j$ annihilates $\eta_j$
  and $A_{j-1} \circ \cdots \circ A_1 h(k_j,x) = \eta_j(x)$ (which
  can be shown using \eqref{eq:etaj}, \eqref{eq:Aj} and \eqref{eq:wronsk2}),
  the product $A_\ell \circ \cdots \circ A_1$ annihilates all
  $h(k_j,x),\; j=1,\ldots, \ell$.  By construction, the $\ell$th order
  operator $A$ has the same property and so must be a multiple of the
  product.  Equality is established by considering the coefficient of
  the highest order derivative $y^{(\ell)}(x)$.

  We prove \eqref{eq:Bfactor} by induction on $\ell$. 
  For $\ell = 1$ the equality of \eqref{eq:Bydef} and \eqref{eq:Bj}  with $j=1$
  follows from \eqref{eq:wronsk1}.
  Suppose that
  the identity \eqref{eq:Bfactor} holds for $\ell-1$.  
  For $1\leq i< j \leq \ell$, let $\tildh_{i,j}$ denote the $\ell-2$
  order Wronskian like \eqref{eq:tildhdef}, but with $h(k_i,x)$ and
  $h(k_j,x)$ omitted.  Since
  \[ \eta_{\ell-1} = \tildh_\ell,\]  we have by \eqref{eq:Bj}, \eqref{eq:wronsk1} and \eqref{eq:wronsk2},
  \begin{align*}
    B_\ell(e^{x^2} \tildh_j) &= \frac{e^{x^2}}{\eta_\ell}
    \Wrop{\tildh_\ell, \tildh_j}= -e^{x^2} \tildh_{j,\ell}.
  \end{align*}
  By the inductive hypothesis,
  \begin{multline*}
    (B_1 \circ \cdots \circ B_{\ell-1}) \left(e^{x^2} \tildh_{j,\ell}\right) \\
    =
        (-1)^{[(\ell-1)/2]} e^{x^2} \eta_{\ell-1}^{-(\ell-1)}
    \Wrop{\tildh_{1,\ell}, \cdots, \tildh_{\ell-1,\ell}, \tildh_{j,\ell}} =
    0, 
    \qquad    j=1,\ldots, \ell-1,
    \end{multline*}
and hence
\begin{align*}
    (B_1 \circ \cdots \circ B_{\ell-1} \circ B_\ell)\left(e^{x^2} \tildh_{j}\right) = 0, \qquad j =1, \ldots, \ell.
  \end{align*}
The last identity also holds for $j= \ell$, since $B_{\ell}(e^{x^2} \tildh_\ell) = 0$.
  
By \eqref{eq:Bydef} we also have
\begin{align*}
    \cB\left(e^{x^2} \tildh_{j}\right) = 0, \qquad j =1, \ldots, \ell.
  \end{align*}
Thus $B_1 \circ \cdots B_{\ell}$ and $\cB$ are two $\ell$-th order linear differential operators that agree for  $\ell$ linearly independent functions
and hence they agree up to a   multiplicative left factor.
  Directly from the definition \eqref{eq:Bj}, we have
  \[ (B_1\circ \cdots \circ B_{\ell})(y) = \frac{1}{\eta_{\ell}}
  y^{(\ell)} + \text{ terms with lower order derivatives.} \] By \eqref{eq:Bydef}, \eqref{eq:wronsk1} and
  \eqref{eq:wronsk3}, we have
  \begin{align*}
    \cB(y) & = (-1)^{\lfloor
      \ell/2\rfloor}\frac{\Wrop{ e^{x^2}\tildh_{1},\ldots,
      e^{x^2}\tildh_{\ell}, y} }{\left( e^{x^2} \eta_\ell\right)^\ell} \\
    & = (-1)^{\lfloor \ell/2\rfloor}\frac{\Wrop{ \tildh_{1},\ldots,
      \tildh_{\ell}} }{\left( 
      \eta_\ell\right)^\ell}  y^{(\ell)} + \text{ terms with lower order derivatives}\\
   & = \frac{1}{\eta_\ell} y^{(\ell)} + \text{ terms with lower order derivatives}.
  \end{align*}
  This establishes \eqref{eq:Bfactor}.
\end{proof}

Next, let
\begin{align*}
  T_j(y)&:=y''-2\left(x+\frac{\eta_j'}{\eta_j} \right) y'+
  \left(\frac{\eta_j''}{\eta_j}+ 2 x \frac{\eta_j'}{\eta_j} \right) y,
\end{align*}
denote the exceptional operators
corresponding to the truncated partition $\lambda^{(j)}$.  Thus $T_0=T$
is the classical Hermite operator, while $T_\ell = \hT$, see \eqref{eqn:hatT}.

\begin{prop}
  \label{prop:ratfac} 
   With $A_j, B_j, T_j$ defined as above, the following intertwining
  relations hold:
  \begin{align}
    \label{eq:Aintertwine}
    A_j \circ T_{j-1}  &=  (T_j -2)\circ A_j, \\
    \label{eq:Bintertwine}
    T_{j-1}\circ B_j &=  B_j\circ (T_j -2).
    \end{align}
\end{prop}
\begin{proof}
For $j=1, \ldots, \ell$, set 
$\psi_j(x) = e^{-x^2/2} h(k_j,x)$
  and observe that
 $\Wrop{\psi_1(x),\ldots, \psi_j(x)}= e^{-j x^2/2} \eta_j(x)$.
  Next, set 
 $\cH_j = -\partial_x^2 + U_j$,
  where 
  \[ U_j= x^2 - 2\partial_x^2 \log \eta_j(x) +2j.\]
The operator $\cH_j$ is such that 
  \begin{equation} \label{eq:HjandTj}
     \left(\frac{e^{-x^2/2}}{\eta_j}\right)^{-1} \circ
     \cH_j \circ \left( \frac{e^{-x^2/2}}{\eta_j}\right) =
     -T_{j}+2j +3.  \end{equation}
       
  
  Since
  \[(-\partial_x^2+x^2) \psi_j = (2k_j+1) \psi_j,\] Crum's Theorem
  \cite{crum} implies that
  \[ \cH_{j-1}\left( \frac{\Wrop{\psi_1,\ldots,
      \psi_j}}{\Wrop{\psi_1,\ldots, \psi_{j-1}}}\right) =
  (2k_{j}+1) \frac{\Wrop{\psi_1,\ldots,
    \psi_j}}{\Wrop{\psi_1,\ldots, \psi_{j-1}}}.
  \]
  Equivalently, in view of \eqref{eq:etaj},
  \[ \cH_{j-1}\left(\frac{e^{-x^2/2} \eta_j}{\eta_{j-1}}\right) =
  (2k_{j}+1)\frac{e^{-x^2/2} \eta_j}{\eta_{j-1}}. \] 
  This gives  the factorizations 
  \begin{align*}
    \cH_{j-1} &= \left(-\partial_x +x -\frac{\eta_j'}{\eta_j} +
      \frac{\eta_{j-1}'}{\eta_{j-1}}\right)\circ \left(\partial_x +x
      -\frac{\eta_j'}{\eta_j} + \frac{\eta_{j-1}'}{\eta_{j-1}}\right)
    + 2k_j+1,\\
    \cH_{j} &= \left(\partial_x +x -\frac{\eta_j'}{\eta_j} +
      \frac{\eta_{j-1}'}{\eta_{j-1}}\right) \circ\left(-\partial_x +x
      -\frac{\eta_j'}{\eta_j} + \frac{\eta_{j-1}'}{\eta_{j-1}}\right)
    + 2k_j+1.
  \end{align*}
  Hence $\cH_j$ is obtained from $\cH_{j-1}$ by a Darboux transformation.
  Straightforward calculations show that
  \begin{align*}
  	A_j & = \frac{\eta_j}{\eta_{j-1}}
  	    \left(\partial_x -  \frac{\eta_j'}{\eta_j} \right) \\
  	    & =  
    \left(\frac{e^{-x^2/2}}{\eta_{j}}\right)^{-1}  
    \circ\left(\partial_x +x -\frac{\eta_j'}{\eta_j} +
      \frac{\eta_{j-1}'}{\eta_{j-1}}\right)\circ \left(
      \frac{e^{-x^2/2}}{\eta_{j-1}}\right), \\
  - B_j &  =\frac{\eta_{j-1}}{\eta_j} \left(-\partial_x +2x+
        \frac{\eta_{j-1}'}{\eta_{j-1}}\right) \\
        & =  \left(\frac{e^{-x^2/2}}{\eta_{j-1}}\right)^{-1}
     \circ\left(-\partial_x +x -\frac{\eta_j'}{\eta_j} +
      \frac{\eta_{j-1}'}{\eta_{j-1}}\right) \circ
    \left(\frac{e^{-x^2/2}}{\eta_{j}}\right).
  \end{align*}  
  Thus, for $1\leq j\leq \ell$,   \begin{align}
    \label{eq:TBAfac}
    T_{j-1} &= B_j \circ A_j+2(j-k_j-1),\\
    \label{eq:TABfac}
    T_j &= A_j\circ B_j+2(j-k_j).
  \end{align}
  Multiplying \eqref{eq:TBAfac}  on the left and \eqref{eq:TABfac} on the right by $A_j$ and equating the terms equal to $A_j\circ B_j\circ A_j$ yields the desired equalities.
\end{proof}

The operators $\cA\circ\cB$ and $\cB\circ\cA$ are polynomials in the operators $\hT$ and $T$ respectively:

\begin{lemma} \label{lem:factor}The products of the operators $\cA$
  and $\cB$ defined in \eqref{eq:Aydef} and \eqref{eq:Bydef} factor
  as:
  \begin{align}
    \label{eq:BApoly}
    \cB\circ\cA&=(T+2k_1)\circ\cdots\circ (T+2k_\ell), \\
    \label{eq:ABpoly}
    \cA\circ \cB&=(\hat T-2\ell+2k_1)\circ \cdots \circ (\hat T-2\ell + 2
    k_{\ell}).
  \end{align}
\end{lemma}
\begin{proof}
It follows from \eqref{eq:TBAfac} that
$$
B_1\circ A_1=T_0-2k_1 = T-2k_1.
$$

Now, suppose that for some $1\leq n<\ell$ that 
\begin{equation}
B_1\circ B_{2}\circ \cdots \circ B_n\circ A_n\circ A_{n-1}\circ\cdots\circ A_1=(T+2k_1)\circ (T+2k_2)\circ \cdots\circ (T+2k_n).\label{eqn:inducthyp}
\end{equation}
Consider the left-hand side of \eqref{eqn:inducthyp} with $n$ replaced by $n+1$.  Using  \eqref{eq:TBAfac} once to replace $B_{n+1}\circ A_{n+1}$ and following repeated use of \eqref{eq:Aintertwine} to move the term involving $T_i$ all the way to the right we obtain 
\begin{align*}
&B_1\circ\cdots\circ B_{n}\circ(B_{n+1}\circ A_{n+1})\circ A_{n}\circ\cdots \circ A_1\\
  &=B_1\circ\cdots\circ B_{n}\circ (T_{n}-2(n-k_{n+1}))\circ A_{n}\circ\cdots \circ A_1\\
  &=B_1\circ\cdots\circ B_{n}\circ A_{n}\circ (T_{n-1}-2(n-k_{n+1}-1))\circ \cdots \circ A_1\\
  &\qquad\qquad\qquad\qquad\vdots\\
  &=B_1\circ\cdots\circ B_{n}\circ A_{n}\circ \cdots \circ
  A_1\circ (T_{0}+2k_{n+1})\\
  &=(T+2k_1)\circ (T+2k_2)\circ \cdots\circ (T+2k_{n+1})
\end{align*}
where the last equality is obtained by making use of the assumption \eqref{eqn:inducthyp} and the fact that $T_0 = T$.  

Then by induction we find that \eqref{eqn:inducthyp} is true when $n=\ell$.  Since the left-hand side in that case is equal to $B\circ A$ by Proposition~\ref{prop:fact}, this proves \eqref{eq:BApoly}.

  Factorization  \eqref{eq:ABpoly} can be similarly obtained by repeated use of \eqref{eq:Bintertwine} and \eqref{eq:TABfac}.
\end{proof}

Combining Lemma~\ref{lem:factor} with the eigenvalue equations \eqref{eqn:Teigen} and \eqref{eq:XDE}  leads immediately to:
\begin{corollary}
  \label{cor:BApi}
  The operators $\cA\circ\cB$ and $\cB\circ\cA$ satisfy the eigenvalue
  equations:
  \begin{align} \label{eq:BAhn}
    (\cB\circ\cA)h(n,x) & =\pi(n) h(n,x), \\
				    \label{eq:ABhhn}
     (\cA\circ \cB)\hh(n,x)& =\pi(n-N+\ell)\hh(n,x).
  \end{align}
  where
  \begin{equation}
    \label{eq:pindef}
    \pi(n)    =\prod_{i=1}^\ell (-2n+2k_i) .
  \end{equation}
\end{corollary}

\section{The Stabilizer and Recurrence Relations}

\begin{definition}\label{def:StabRing} Let us define the stabilizer
  ring $\StabRing\subset\C[x]$ by
\[
\StabRing=\left\{\stabf\in\C[x]\ : \stabf(x)\hh(n,x)\in
  \cU\hbox{ for all }n\right\}.
\]
\end{definition}
\begin{remark}
  By the orthogonality properties of the polynomials $\hh(n,x)$ it is
  possible to show that if $\stabf\in \StabRing$, then for all
  $n\in \Ilam$, necessarily
  \[ f(x) \hh(n,x) =\sum_{j=-\deg \stabf}^{\deg\stabf}\gamma_{j,n}\hat
  h(n+j,x) \] for some $\gamma_{j,n}\in \C$.
\end{remark}

It was already observed in \cite{Sasaki2010} that $\eta^2
(x)\in\StabRing$ and more recently in
\cite{Miki2015,Duran2015a,Odake2015} that $\int^x \eta\in\StabRing$ as
well.  The following proposition provides a characterization of
$\StabRing$.
\begin{prop}\label{prop:stab}
  If $f(x)\in\C[x]$ and $f'(x)$ is divisible by $\eta(x)$, then $f\in\StabRing$.  The
converse is also true, provided $\eta(x)$ has only simple roots.
\end{prop}

\begin{proof} 
We first show that $p\in
  \C[x]$ belongs to $\cU$ if and only if
  \begin{equation}
    \label{eq:Tlambdasing} 2 \eta'(x)(xp(x) -p'(x)) +\eta''(x) p(x)
  \end{equation} is divisible by $\eta$. Indeed, the operator  \eqref{eqn:hatT}  transforms every element  $p\in\cU$
  into a polynomial, so the numerator of the singular part of $\hat T(p)$ must be divisible by $\eta$: this is precisely \eqref{eq:Tlambdasing}.
Conversely, divisibility  of \eqref{eq:Tlambdasing} by $\eta$ imposes precisely $\deg \eta = N$  independent linear conditions on $p\in \C[x]$.
  Since $N$ is also the codimension of $\cU$ it
  follows by dimensional exhaustion that a polynomial $p\in \C[x]$ that
  satisfies these conditions is necessarily an element of
  $\cU$.  
  
  Now suppose that $f'(x)$ is divisible by $\eta(x)$. Let $p\in
  \cU$ and set $q=fp$.  Then,
  \begin{align*}
    &2 \eta'(x)(xq(x) -q'(x)) +\eta''(x) q(x) = \\
    &\qquad     f(x)(2 \eta'(x)(x p(x) -p'(x)) +\eta''(x) p(x))  - 2 \eta'(x) p(x) f'(x)
  \end{align*}
  is divisible by $\eta$, and hence is also in $\cU$.

  Conversely, suppose that $\eta(x)$ has simple zeros, and that the
  above expression is divisible by $\eta(x)$ for all $p \in
  \cU$.  Also, $p(x) f'(x)$ is divisible by $\eta(x)$ for all
  $p\in \cU$.  By Proposition 5.5 in \cite{Gomez-Ullate2014a},
  the polynomials in $\cU$ do not have a common root, which
  implies that $f'(x)$ must be divisible by $\eta(x)$.
\end{proof}

Note that as a consequence of Proposition~\ref{prop:stab} the ring
$\StabRing$ is not only non-empty but moreover guaranteed to have
elements of every sufficiently high degree.  Hence if $\eta(x)$ has
simple zeros, then $\int^x \eta\in\StabRing$ is the element in
$\StabRing$ of minimal degree.\footnote{The minimal degree of an element of $\StabRing$ and thus the minimal order of the recurrence relation for the exceptional Hermite polynomials was conjectured to be $2N+3$ in the previous works \cite{Miki2015,Duran2015a,Odake2015}. We see that this conjecture on the minimal order relies on Veselov's conjecture on the simplicity of the zeros of the Wronskian of Hermite polynomials \cite{Felder2012}.  }

The elements of $\StabRing$ can also be characterized as precisely
those polynomials in $\C[x]$ that act as linear endomorphisms on
$\cU$.   From   \cite{Gomez-Ullate2014a} we have
\begin{lemma}
  The operators $A,B$ defined in \eqref{eq:Aydef} and \eqref{eq:Bydef}
  act as linear transformations $A:\C[x] \to \cU$ and $B:\cU
  \to\C[x]$, respectively.
\end{lemma}
\noindent
An immediate consequence is the following.
\begin{lemma} \label{lem:BfA} 
	For $\stabf\in\StabRing$, $\cB\circ \stabf \circ \cA$ is
  a differential operator in $\C[x,\partial]$.  
\end{lemma}
\begin{proof}
  By definition of $\StabRing$ and the preceding lemma, applying
  $\cB\circ f\circ \cA$ to any polynomial results in a polynomial.
  Applying it to $1$ then reveals that the order zero term is a
  polynomial.  The claim then follows by induction since applying the
  operator to $x^j$ reveals that the coefficient of $\partial^j$ is a
  linear combination of powers of $x$ multiplied by the lower order
  coefficients and the result, all of which are polynomial by
  assumption.
\end{proof}

The reason we need the previous lemma is so that we can apply the
anti-isomorphism $\flat$ to $\cB\circ\stabf\circ \cA$:
\begin{definition}\label{def:Deltaf} For $\stabf\in\StabRing$, let 
  \begin{equation}
    \label{eq:Deltaf}
    \Deltahat=\shift^{\ell-N}\circ\flat(\cB\circ \stabf\circ \cA)\circ
    \frac{1}{\pi(n)}\circ \shift^{N-\ell}
  \end{equation}
\end{definition}

\begin{lemma}
  \label{lem:Binj}
  The linear transformation $B:\cU\to \C[x]$ is injective.
\end{lemma}
\begin{proof}
  By \eqref{eq:ABhhn},
  \[ (A\circ B)\hh(n,x) = 2^\ell\left( \prod_{i=1}^\ell( k_i + N-\ell
    - n)\right)\hh(n,x).\] By \eqref{eq:Ilamdef}, if $n\in \Ilam$ is
  one of the permitted degrees, then the right side is not zero.  Hence, none of the basis elements of $\cU$ can be in the kernel of $B$.
\end{proof}

Our main result is that the differential operator $\Deltahat$ has
$\hh(n,x)$ as eigenfunction with eigenvalue $\stabf$:

\begin{theorem}\label{thm:Deltaf} For all $\stabf\in\StabRing$,
  $\Deltahat\hh(n,x)=\stabf(x)\hh(n,x)$.\end{theorem}
\begin{proof}
  Let $\stabf \in \StabRing$ and set 
  \[ \tDelta = \flat(\cB\circ \stabf\circ \cA)\circ
    \frac{1}{\pi(n)} .\]
    Hence, applying $\tDelta\circ \pi(n)$ to $h(n,x)$ and invoking
    Corollary \ref{cor:BApi}  gives
  \begin{align} \nonumber
    (\tDelta\circ \pi(n))(h(n,x)) &= (\tDelta \circ B\circ
    A)(h(n,x))\\ \nonumber
    &= ( B \circ\tDelta\circ A)(h(n,x))\\
    &= ( B \circ\tDelta)(\hh(n+N-\ell,x))  \label{eq:Theorem510a}
  \end{align}
  Furthermore,
  \begin{align} \nonumber
    \flat(B\circ f\circ A)(h(n,x)) &= (B\circ f\circ A)(h(n,x)) \\
    &= (B\circ f)(\hh(n+N-\ell,x)) \label{eq:Theorem510b}
  \end{align}
 Since $\tDelta\circ\pi(n)=\flat(B\circ f \circ A)$, the
expressions \eqref{eq:Theorem510a} and \eqref{eq:Theorem510b} are equal 
and so the operator $B$ annihilates  the polynomial
  \[\tDelta(\hh(n+N-\ell,x)) - f(x) \hh(n+N-\ell,x).\]
  For each fixed $n$, the first term is a linear combination of polynomials in
  $\cU$ and the second term is in $\cU$, by assumption.
 Since no nonzero elements of $\cU$ are in the kernel of $B$ (see Lemma~\ref{lem:Binj}), this must be the zero polynomial.  
Finally, note that
  \begin{align*}
    \Deltahat \hh(n,x) &= (\shift^{\ell-N}\circ\tDelta\circ \shift^{N-\ell})
    \hh(n,x) 
    = \shift^{\ell-N}(\tDelta \hh(n+N-\ell,x))\\
    &= \shift^{\ell-N}(f(x)  \hh(n+N-\ell,x))= f(x)  \hh(n,x),
  \end{align*}
	as claimed in Theorem \ref{thm:Deltaf}.
\end{proof}

\section{Examples}

In this section we illustrate Theorem \ref{thm:Deltaf} by exhibiting a
closed form description of the recurrence relation including some
specific examples as well as some very general ones. 

\subsection{One-step examples} Let us first
consider the case of $\ell=1$  and write $K=\{ k \}$ where $k\geq 1$ is
an integer.   By Proposition \ref{prop:stab} and the fact that $H_{k+1}' = 2(k+1) H_{k}$,
we have that $h(k+1,x) \in \StabRing$, which leads us to the following identity. 

\begin{prop}  
  \label{prop:1steprr}
  Fix $k\geq 1$ and write
  \[ \hh(n,x) = \Wrop{h(k,x),h({n-k+1},x)},\quad n\geq k-1,\; n\neq 2k-1 \] 
	for the corresponding 1-step exceptional Hermite polynomial of degree $n$.  Then,
  \begin{multline}
    h({k+1},x) \hh(n,x) \\
     = (n-2k+1) \sum_{j=0}^{k+1} 2^j
    \binom{k+1}{j} (n+3-k-j)_{j-1} \hh({n+k+1-2j},x)
  \end{multline}
  where
  \[ (X)_n = X (X+1) \cdots (X+n-1) \]
  denotes the usual Pochhammer symbol.
\end{prop}
The proof of Proposition \ref{prop:1steprr} requires the following two lemmas.
\begin{lemma}
  Setting $f(x)=h({k+1},x)$ and
  \begin{align*}
    B(y(x)) &= (y'-2x y)/h({k},x)\\
    A(y(x)) &= \Wrop{h(k,x),y(x)}
  \end{align*}
  as per the definitions above, we have
  \[ (B\circ f\circ A)[y] = h({k+1},x) y'' - h({k+2},x) y' +
  2k(h({k+1},x)-2(k+1)h({k-1},x))y \]
\end{lemma}
\begin{proof}
  A direct calculation suffices to establish this.
\end{proof}

\begin{lemma}
  For integers $k,n\geq 0$ we have
  \[ h(k,\Delta) = \sum_{j=0}^k 2^j \binom{k}{j} (n-j+1)_j\circ
  \shift^{k-2j},\] where $(n-j+1)_j$ is the multiplication operator by
  the indicated polynomial in $n$.
\end{lemma}
\begin{proof}
  This is just the restatement of the linearization formula for the
  product of Hermite polynomials \cite[p. 42]{Askey75}
\end{proof}
\begin{proof}[Proof of Proposition \ref{prop:1steprr}]
  By the first lemma,
  \begin{align*}
    \flat(B\circ f\circ A) &= \Gamma^2 \circ h({k+1},\Delta) -
    \Gamma\circ h({k+2},\Delta) + 2k(h({k+1},\Delta) -2(k+1)h({k-1},\Delta))    
  \end{align*}
  Using the second lemma and elementary algebraic manipulations gives
  \begin{align*}
    \flat(B\circ f\circ A) &= (k-n) \sum_{j=0}^{k+1}
    2^{j+1}(n-2j+1) \binom{k+1}{j} (n-j+2)_{j-1} \shift^{k+1-2j}
  \end{align*}
  Now 
  \[ \pi(n) = 2(k-n).\] Applying the Definition~\ref{def:Deltaf}
  amounts to division of the preceding expression by $-2(n-2j+1)$ and
  gives the desired formula.
\end{proof}

\subsection{Multi-step examples}
Next, we rederive the Miki-Tsujimoto-Dur\'an relation
\cite{Duran2015a,Miki2015}.  In this case, $\lambda = (1,1)$ and $K =
\{1,2\}$, and
\begin{equation} \label{eq:pinMTD}
  \pi(n) = 4(n-1)(n-2).
	\end{equation}
We have
\begin{align}	\label{eq:etaMTD}
  \eta(x) & = \Wrop{h(1,x), h(2,x)} =  4(1+2x^2) \\
  \hh(n,x) & = \Wrop{h(1,x),h(2,x),h(n,x)}. 
\end{align}
In order to avoid the use of fractions, we take
\begin{align} \label{eq:fMTD}
	f(x) =  \frac{3}{4} \int_0^x \eta(s) ds = x(3+2x^2) \in \StabRing. 
	\end{align}
	
Applying \eqref{eq:Aydef} and \eqref{eq:Bydef} gives
\begin{align*}
  A(y)  &= 4 ( (1+2x^2) y'' -4 xy' + 4 y)\\
  B(y) &= \frac{1}{4(1+2x^2)}\, y'' - \frac{ 2x(1+x^2)}{(1+2x^2)^2}\,  y' +   \frac12\, y.
	\end{align*}
Further calculations using this and \eqref{eq:fMTD} then show that
\begin{multline} \label{eq:BfAMTD}
 (B\circ f \circ A)(y) = x(3+2x^2) \, y'''' + (6-8x^4) \,   y''' \\
  - x(6 + 8x^2 - 8x^4) \, y'' + x^2(24-16x^2) \, y' - x(24 -16x^2) \, y 
	\end{multline}
which has indeed only polynomial coefficients, as claimed in Lemma \ref{lem:BfA}.

Then applying $\flat$ according to our Definition \ref{def:flat}, we find
\begin{multline*} 
  \flat(B\circ f \circ A) = \Gamma^4 \Delta (3+2\Delta^2)  + \Gamma^3(6-8\Delta^4) \\ 
		- \Gamma^2 \Delta (6\Delta + 8 \Delta^2 - 8 \Delta^4) + 8\Gamma \Delta^2(24 - 16 \Delta^2)
			- \Delta (24 - 16 \Delta^2). 
			\end{multline*}
We rewrite this in terms of powers of the shift operator $\shift$ by means of the identities
\eqref{eq:Delta}--\eqref{eq:Gamma}. This is a lengthy calculation (which is easier to do with the
help of symbolic software) and it  results in
	\begin{align} \label{eq:flatBfAMTD}
  \flat(B\circ f \circ A) 
  = (n-2)_2 \shift^3 + 6 (n-2)_3 \shift + 12 (n-3)_4 \shift^{-1} + 8
  (n-5)_2(n-2)_3 \shift^{-3}.
	\end{align}
Note that we use Pochhammer symbols.	

By Definition \ref{def:Deltaf} and \eqref{eq:pinMTD} we have (since $N= \ell = 2$ in this example)
\begin{align}
  4 \Deltahat &=  \flat(B\circ f \circ A)  \circ \frac{1}{(n-2)_2} \\
	& = 	\frac{(n-2)_2}{(n+1)_2} \shift^3 + 6 (n-2) \shift + 12
  (n-1)_2 \shift^{-1} + 8 (n-2)_3 \shift^{-3}
\end{align}
where we used $ \shift^k \circ \, a(n) = a(n+k) \, \shift^k$
and  simplifications of the Pochhammer symbols.
We therefore recover the desired recurrence relation
\begin{multline}   \label{eq:mikitsuj}
  4 x(3+2x^2)  \hh(n,x) = 4 \Deltahat \hh(n,x)  \\
	= \frac{(n-2)_2}{(n+1)_2} \hh({n+3},x) + 6 (n-2)
  \hh({n+1},x) \\ 
  + 12 (n-1)_2 \hh({n-1},x) + 8 (n-2)_3 \hh({n-3},x).
\end{multline}

The above procedure is entirely algorithmic and can be easily
implemented using a computer algebra system. Here, for example, is the
2-step relation corresponding to $\lambda=(2,2)$.
Let
\[
\hh(n,x) = \Wrop{h(2,x),h(3,x),h({n-2},x)}
\]
be the corresponding exceptional Hermite polynomial of degree $n$.  Then
\begin{align*}
  &\left(24x + \frac{32}{5} x^5\right) \hh(n,x) = \frac15
  \frac{(n-5)_2}{(n)_2} \hh({n+5},x)+ 2\frac{(n-5)_2}{n-1} \hh({n+3},x)
  \\&\qquad +8(n-5)(n-3) \hh({n+1},x)+ 16 (n-2)^2(n-4) \hh({n-1},x)\\
  &\qquad + 16(n-5)_4\hh({n-3},x)+ \frac{32}{5}
  (n-6)_5 \hh({n-5},x)
\end{align*}

Here is a 4-step example with $\lambda = (1,1,2,2)$.  Let
\[ \hh(n,x) = \Wrop{h(1,x), h(2,x), h(4,x), h(5,x), h({n-2},x)} \]
be the corresponding degree $n$ exceptional polynomial. Then,
\begin{align*}
  &(105 x  +70 x^3 +84 x^5+24 x^7) \hh(n,x) =\\
  &\qquad \frac{3}{16}\frac{(n-7)_2(n-4)_2}{(n)_2 (n+3)_2}
  \hh({n+7},x) + \frac{21}{8}\frac{(n-7)_2(n-4)_2}{(n-1)(n+1)_2}  
  \hh({n+5},x)\\
  &\qquad + \frac{7}{4}\frac{(n-7)_2 (80-57n+9n^2)}{(n-1)_2} \hh({n+3},x) +
  \frac{105}{2} (n-7)_2(n-4)\hh({n+1},x)\\
  &\qquad +105 (n-7)_2(n-3)_2 \hh({n-1},x)+14
  (n-4)_3(332-111n+9n^2)\hh({n-3},x) \\
  &\qquad + 84 (n-7)_6 \hh({n-5},x)+24(n-8)_7\hh({n-7},x)
\end{align*}
\section{Closing Remarks}
Connections have already been recognized between the well-established
theory of bispectrality and the more recent subject of
exceptional orthogonal polynomials.  However, as this note
demonstrates, these connections have not been fully utilized.  Although progress had been made recently on the question of recurrence relations for the exceptional orthogonal polynomials, their existence and precise form remained a difficult problem.  The use of the bispectral anti-isomorphism reduces the problem to very simple (nearly trivial) algebraic manipulation.  

Certainly, there are more opportunities for a fruitful exchange of
ideas between these two closely related fields of research that have
not yet been realized. 
In future publications, we hope to address the structure of $\{\Deltahat:\stabf\in\StabRing\}$ as a commutative ring of difference operators and its corresponding spectral curve (cf.\ \cite{Wilson1993}) as well as the implications of the ad-nilpotency of its elements (cf.\ \cite{Duistermaat1986}).

\section*{Acknowledgements}
The authors are grateful to the organizers of the conference NEEDS
2015 in Sardinia, Italy without which we might never have recognized
the potential application of these tools to this problem.  Likewise,
the authors are grateful to Yves Grandati and Satoshi Tsujimoto for
helpful conversations and their presentations on related topics at
that conference.

D.G.U. has been supported in part by Spanish MINECO Grants \mbox{MTM2012-31714} and \mbox{FIS2012-38949-C03-01} and by the ICMAT-Severo Ochoa grant \mbox{SEV-2011-0087}. A.B.J.K. is supported by KU Leuven Research Grant OT/12/073,
the Belgian Interuniversity Attraction Pole P07/18, and FWO Flanders projects G.0641.11 and G.0934.13.  
R.M. is supported by NSERC grant RGPIN-228057-2004.

\end{document}